\documentclass[12pt]{amsart}

\usepackage{graphicx}
\usepackage{latexsym}
\usepackage{amsfonts}
\usepackage{amscd}
\usepackage{amssymb}
\usepackage{amsmath}
\usepackage{amsthm}
\usepackage{bm}
\usepackage[all]{xy}
\usepackage[lmargin=1.25in,rmargin=1.25in,tmargin=1.25in,bmargin=1in,dvips]{geometry}
\usepackage{pinlabel}
\usepackage{paralist}
\usepackage{mathtools}
\usepackage{enumitem}
\usepackage[normalem]{ulem}
\usepackage[unicode, linktocpage, bookmarksnumbered=true, bookmarksdepth=3]{hyperref}
\usepackage{tikz}
\usepackage{subfigure}

\usepackage[textsize=scriptsize]{todonotes}

\newtheorem{theorem}{Theorem}[section]
\newtheorem{lemma}[theorem]{Lemma}

\newtheorem{proposition}[theorem]{Proposition}

\theoremstyle{definition}

\newtheorem{remark}[theorem]{Remark}
\newtheorem{question}[theorem]{Question}

\newtheorem{example}[theorem]{Example}

\def\N{\mathbb{N}}
\def\R{\mathbb{R}}

\def\Z{\mathbb{Z}}
\def\Q{\mathbb{Q}}

\def\RP {\mathbb{R} \mathrm{P}}
\def\CP {\mathbb{C} \mathrm{P}}

\def\CC {\mathcal{C}}

\def\KK {\mathcal{K}}

\def\SS {\mathcal{S}}

\newcommand{\abs}[1] {\left\lvert #1 \right\rvert}

\def\co{\colon\thinspace}

 \DeclareMathOperator{\intt}{int}

\def\conn{\mathbin{\#}}
\def\bconn{\mathbin{\natural}}

\renewcommand{\MR}[1]{}

\definecolor{darkgreen}{rgb}{0,0.5,0}
\definecolor{purple}{rgb}{0.5,0,0.5}

\def\CP{\mathbb{C} \mathrm{P}}
\def\CPbar{\overline{\mathbb{C} \mathrm{P}}{}}

\numberwithin{equation}{section}

\begin{document}

\title{A note on rationally slice knots}
\author{Adam Simon Levine}
\address{Department of Mathematics, Duke University, Durham, NC 27708}
\email{alevine@math.duke.edu}
\thanks{The author was supported by NSF grant DMS-2203860.}

\begin{abstract}
Kawauchi proved that every strongly negative amphichiral knot $K \subset S^3$ bounds a smoothly embedded disk in some rational homology ball $V_K$, whose construction \emph{a priori} depends on $K$. We show that $V_K$ is independent of $K$ up to diffeomorphism. Thus, a single 4-manifold, along with connected sums thereof, accounts for all known examples of knots that are rationally slice but not slice.
\end{abstract}

\maketitle

\section{Introduction} \label{sec: intro}

Let $K$ be a knot in $S^3$. If $X$ is a smooth, compact, oriented $4$-manifold with boundary $S^3$, we say that $K$ is \emph{slice in $X$} if there exists a smoothly embedded disk $D$ in $X$ with boundary equal to $K$. Note that if $K$ is slice in $X$, then so is any knot that is smoothly concordant to $K$.

For a commutative ring $R$ with unit, we say that $K$ is \emph{$R$-slice} if it is slice in some $4$-manifold $X$ that is an $R$-homology $4$-ball. We will focus on the cases of $R = \Z$, $\Q$, and $\Z_p$ (for $p$ prime). Note that a $\Z_p$-homology $4$-ball $X$ is the same as a $\Q$-homology $4$-ball with the additional property that $\abs{H_1(X;\Z)}$ is not divisible by $p$. We use \emph{rationally slice} as a synonym for $\Q$-slice.\footnote{Some authors, e.g. Kawauchi \cite{KawauchiRational}, impose an additional homological constraint in the definition of \emph{rationally slice}, and use \emph{weakly rationally slice} for the definition we are using. Our terminology agrees with that of other recent papers on the subject, e.g. \cite{KimWuRational, HomKangParkStoffregen}, which use \emph{strongly rationally slice} for Kawauchi's version.}

By a slight abuse of notation, if $Z$ is a closed $4$-manifold and $K$ is slice in $Z - B^4$, we also say that $K$ is slice in $Z$. If $X = Z - B^4$, then $X$ is an $R$-homology $4$-ball if and only if $Z$ is an $R$-homology $4$-sphere.

Let $\CC$ denote the smooth concordance group, and let $\KK_R$ denote the subgroup of $\CC$ consisting of concordance classes of knots that are $R$-slice. In other words, $\KK_R$ is the kernel of the forgetful map $\CC \to \CC_R$, where $\CC_R$ is the group of knots in $S^3$ up to concordance in $R$-homology cobordisms. 

It remains an open question whether there exist knots that are $\Z$-slice but not slice, i.e. whether $\KK_\Z \ne 0$. In contrast, it is well-known that there exist knots that are $\Q$-slice but not slice (or even $\Z$-slice), such as the figure-eight knot. Specifically, a knot $K \subset S^3$ is called \emph{strongly negative amphichiral} if there exists an orientation-reversing involution $\phi \co S^3 \to S^3$ preserving $K$ setwise and having exactly two fixed points, both lying on $K$. Following terminology of Keegan Boyle, we refer to a \textbf{s}trongly \textbf{n}egatively \textbf{a}mphi\textbf{c}hiral \textbf{k}not as a SNACK. Note that every SNACK represents a class of order at most 2 in $\CC$. Kawauchi \cite{KawauchiFigureEight, KawauchiRational} showed that every SNACK is $\Q$-slice; more precisely, he proved that every SNACK $K$ is slice in a certain rational homology $4$-ball $V_K$, whose construction \emph{a priori} depends on $K$.

The main theorem of this note is that $V_K$ is in fact independent of $K$ up to diffeomorphism; that is, all SNACKs are slice in the same rational homology 4-ball. We may describe the manifold explicitly as follows. Let $\tau \co S^2 \times S^2 \to S^2 \times S^2$ be the map $\tau(x,y) = (r(x), -y)$, where $r \co S^2 \to S^2$ is a reflection. This map is an orientation-preserving involution with no fixed points, so the quotient $Z_0 = S^2 \times S^2 / \tau$ is a closed, orientable manifold. Some elementary algebraic topology (see Lemma \ref{lemma: ZK-algtop} below) shows that $Z_0$ is a rational homology 4-sphere with $\pi_1(Z_0) \cong H_1(Z_0) \cong H_2(Z_0) \cong \Z_2$. Thus, for every odd prime $p$, $Z_0$ is a $\Z_p$-homology sphere. Note that the map $(x,y) \mapsto (x,-y)$ induces an orientation-reversing involution on $Z_0$.

In Section \ref{sec: kawauchi}, we will prove:

\begin{theorem} \label{thm: main}
For every strongly negative amphichiral knot $K \subset S^3$, Kawauchi's manifold $V_K$ is diffeomorphic to $Z_0 - B^4$. Thus, every SNACK is slice in $Z_0 - B^4$.
\end{theorem}

\begin{figure}
\labellist
 \pinlabel $2$ [b] at 16 58
 \pinlabel $0$ [bl] at 49 65
 \pinlabel {$\cup \, h_3, h_4$} [l] at 86 32
\endlabellist
\includegraphics{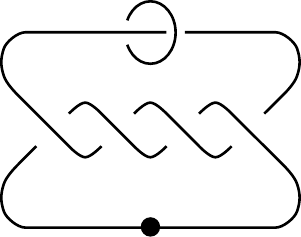}
\caption{Kirby diagram for $Z_0$.} \label{fig: Z0-kirby}
\end{figure}

\begin{remark}
For another characterization of $Z_0$, consider the map $q\co S^2 \times S^2 \to \RP^2$ taking $(x,y)$ to the class of $y$. Then $q \circ \tau = q$, so $q$ descends to a map $\bar q \co Z_0 \to \RP^2$, which gives $Z_0$ the structure of an $S^2$-bundle over $\RP^2$. If $x \in S^2$ is any fixed point of the reflection $r$, we obtain a section $\sigma_x \co \RP^2 \to Z_0$ by defining $\sigma_x([y]) = [(x,y)]$ for each $y \in S^2$. Since the fixed-point set of $r$ is a circle, we in fact find a $1$-dimensional family of nearby disjoint sections. The manifold $Z_0$ is thus characterized by being the unique $S^2$ bundle over $\RP^2$ with orientable total space and a section of self-intersection $0$. (See \cite[p.~237]{HillmanGeometries} for further discussion of $S^2$-bundles over $\RP^2$.)

We claim that $Z_0$ is represented by the handle diagram in Figure \ref{fig: Z0-kirby} (using dotted $1$-handle notation). As seen in \cite[Figure 6.2]{GompfStipsicz}, the $0$-handle, $1$-handle, and $2$-framed $2$-handle from the figure produce the $D^2$-bundle over $\RP^2$ with orientable total space and Euler number $0$. The double of that $D^2$-bundle is the $S^2$-bundle described above, which is $Z_0$. We obtain the double by adding a $0$-framed $2$-handle along the meridian of the first $2$-handle, and then a $3$-handle and $4$-handle, which yields Figure \ref{fig: Z0-kirby}.
\end{remark}

Before we turn to the proof of Theorem \ref{thm: main}, we discuss its implications for the study of rationally slice knots, albeit with more questions than answers. Surprisingly, Kawauchi's construction actually accounts for all known examples of knots that are $\Q$-slice but not slice, that is, all known elements of $\KK_\Q$. We make this explicit as follows.

First, note that if $K$ and $K'$ are knots, and if $K$ is slice in a 4-manifold $X$ and $K'$ is slice in $X'$, then $K\conn K'$ is slice in $X \bconn X'$, and $-K$ is slice in $\overline{X}$ (i.e. $X$ with reversed orientation). Let $\SS$ denote the set of concordance classes of knots that are slice in $\bconn n (Z_0-B^4)$ for some $n \in \N$ (or equivalently in $\conn n Z_0)$. Because $Z_0 \cong \overline{Z_0}$ as oriented manifolds, we thus see that $\SS$ is a subgroup of $\CC$ and is contained in $\KK_\Q$. Indeed, for every odd prime $p$, we have $\SS \subset \KK_{\Z_p}$.

For any knots $P \subset S^1 \times D^2$ and $K \subset S^3$, let $P(K)$ denote the satellite knot with pattern $P$ and companion $K$ (i.e. the image of $P$ under the embedding $S^1 \times D^2 \to S^3$ determined by the 0-framing of $K$). The operation $K \mapsto P(K)$ descends to a function on $\CC$. If $P(O)$ is slice (where $O$ denotes the unknot\footnote{Many authors use $U$ to denote the unknot, but we prefer $O$ because of the obvious graphical similarity.}), we call $P$ a \emph{slice pattern}, and the operation $K \mapsto P(K)$ a \emph{slice satellite operation}. If a knot $K$ is slice in a particular $4$-manifold $X$, then so is $P(K)$ for any slice pattern $P$; thus, the subgroups $\KK_R$ (for any ring $R$) and $\SS$ are closed under slice satellite operations.

To the author's knowledge, the only known concordance classes of knots that are rationally slice but not slice (that is, nontrivial elements of $\KK_\Q$) arise from Kawauchi's construction, together with taking iterated slice satellite operations and/or connected sums, and thus they lie in $\SS$. We note several such constructions in the literature:

\begin{itemize}
\item
Cha \cite[Theorem 4.16]{ChaRational} exhibited an family of SNACKs that generate a $\Z_2^\infty$ subgroup of $\KK_\Q$. These knots can be distinguished up to concordance by their classes in the algebraic concordance group \cite{JLevineCobordism}. Subsequently, Hedden, Kim, and Livingston \cite{HeddenKimLivingston} found another such family of SNACKS with the additional property of being topologically slice (and hence algebraically slice). The proof that these knots fail to be slice relies on the Heegaard Floer $d$ invariants \cite{OSzAbsolute} of the knots' branched double covers.

\item
For a knot $K$ and relatively prime integers $m,n$, let $K_{m,n}$ denote the $(m,n)$ cable of $K$ (where $m$ denotes the winding number in the longitudinal direction and $n$ in the meridional direction). Then for any $m \in \Z$, the operation $K \mapsto K_{m,1}$ is a slice satellite operation. Let $F$ denote the figure-eight knot, which is strongly negative amphichiral and hence slice in $Z_0$. Hom, Kang, Park, and Stoffregen \cite{HomKangParkStoffregen} proved that the set of knots
\[
\{F_{2n-1, 1} \mid n \ge 2\}
\]
is linearly independent in $\CC$, and thus generates a $\Z^\infty$ subgroup of $\CC$ contained in $\KK_\Q$ (and indeed in $\SS$).\footnote{The result is stated in \cite{HomKangParkStoffregen} for $F_{2n-1,-1}$, but note that $F_{2n-1,-1}$ is the mirror of $F_{2n-1,1}$.} This result provided the first known non-torsion elements of $\KK_\Q$. For linear combinations consisting of more than one summand, the resulting knot is slice in some connected sum of copies of $Z_0$, but \emph{a priori} not necessarily slice in $Z_0$ itself. The proof makes use of concordance invariants coming from involutive knot Floer homology \cite{HendricksManolescuInvolutive}.

More recently, Dai, Kang, Mallick, Park, and Stoffregen \cite{DKMPSCable}, answering a long-standing question of Kawauchi \cite{KawauchiFigureEight}, proved that $F_{2,1}$ is not slice (and indeed generates a $\Z$ subgroup of $\SS$). This proof relies on using the involutive structure of the Heegaard Floer homology of the branched double cover and the action of the deck transformation.

\item
Kawauchi's result applies only to strongly negative amphichiral knots, but not necessarily to knots that are merely negative amphichiral (isotopic to their mirror reverses). However, Kim and Wu \cite{KimWuRational} proved that if $K$ is a fibered, negative amphichiral knot whose Alexander polynomial is irreducible, then $K$ is necessarily obtained from a SNACK by iterated slice satellite operations, and hence is rationally slice by Kawauchi's result. Again, any such knot must lie in $\SS$.
\end{itemize}

In some sense, Theorem \ref{thm: main} illustrates how little is known about rational concordance: a single $4$-manifold (along with connected sums of copies thereof) accounts for all known examples of knots that are rationally slice but not slice. That is, the following question is open:

\begin{question} \label{q: KQ-SZ0}
Is $\KK_\Q = \SS$? That is, is every rationally slice knot slice in a connected sum of copies of $Z_0$?
\end{question}

To try to answer Question \ref{q: KQ-SZ0} in the negative, it is instructive to consider not only $\Q$-concordance but also $\Z_p$-concordance. By the above discussion, all known elements of $\KK_\Q$ are contained in $\KK_{\Z_p}$ for every odd prime $p$. In contrast, the following question remains open:
\begin{question} \label{q: KZ2}
Is $\KK_{\Z_2} \ne 0$? That is, does there exist a knot $K \subset S^3$ that is slice in a $\Z_2$-homology ball but not slice?
\end{question}
In some sense, this question is nearly as difficult as that of the better-known problem of finding nontrivial elements of $\KK_\Z$. A large number of knot invariants, including Heegaard Floer invariants such as $\tau$ \cite{OSz4Genus} and $\Upsilon$ \cite{OSSzHomomorphisms}, necessarily vanish for all rationally slice knots. Most crucially, even the invariants used in the above-mentioned results, which can detect some nontrivial elements of $\KK_\Q$, are unable to obstruct a knot from being slice in a $\Z_2$-homology 4-ball. Namely, if a knot $K$ is $\Z_2$-slice, then:
\begin{itemize}
\item it is algebraically slice \cite[Theorem 3]{ChaLivingstonRuberman};

\item the slice obstructions from involutive knot Floer homology vanish \cite[Remark 1.8]{HomKangParkStoffregen}; and

\item the branched double cover of $K$ bounds a $\Z_2$-homology ball, and hence the obstructions from $d$ invariants and involutive Floer homology vanish \cite[Remark 5.4]{DKMPSCable}.
\end{itemize}
It remains unknown whether Rasmussen's $s$ invariant \cite{RasmussenGenus} (or any of its generalizations) vanishes for all rationally slice knots.

Nevertheless, here is one potential approach to Questions \ref{q: KQ-SZ0} and \ref{q: KZ2}. First, recall that if a knot $K$ is slice in a $\Z_p$-homology ball $X$, then for any power $p^k$, the $p^k$-fold cyclic branched cover of $X$ branched over the slice disk is again a $\Z_p$-homology ball whose boundary is $\Sigma_p(K)$. (On the other hand, if $H_1(X;\Z_p) \ne 0$, then this covering may not be a rational homology ball.) Thus, suppose one can find a knot $K$ that is slice in a $\Z_2$-homology $4$-ball $X$ that is not an integer homology ball, and choose any odd prime $p$ dividing $\abs{H_1(X;\Z)}$. If one can show that $\Sigma_{p^k}(K)$ does not bound any rational homology ball (using, say, $d$ invariants), it then follows that $K$ cannot be $\Z_p$-slice, and in particular it cannot be in $\SS$. This would thus resolve both Question \ref{q: KQ-SZ0} (in the negative) and Question \ref{q: KZ2} (in the affirmative).

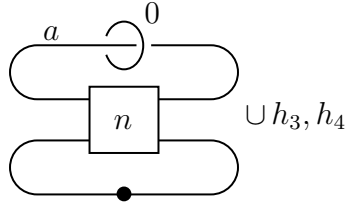
\begin{figure}
\labellist
 \pinlabel $a$ [b] at 16 58
 \pinlabel $0$ [bl] at 49 65
 \pinlabel $n$ at 43 30
 \pinlabel {$\cup \, h_3, h_4$} [l] at 86 32
\endlabellist
\includegraphics{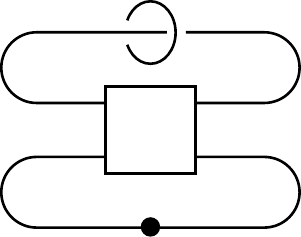}
\caption{Kirby diagram for $X_{n,a}$. The box indicates $n$ full positive twists.} \label{fig: Xna-kirby}
\end{figure}

Klug and Ruppik \cite[Corollary 2.5]{KlugRuppikDeep} proved that if $X$ is a closed $4$-manifold whose universal cover $\tilde X$ is $\R^4$ or $S^4$, then any knot that is slice in $X$ is slice. However, this is not an issue if $X$ is a rational homology $4$-sphere with finite (nontrivial) fundamental group; a simple Euler characteristic argument shows that the universal cover must have nontrivial $H_2$. For instance, for any $n \in \N$ and $a \in \Z$, let $X_{n,a}$ denote the closed 4-manifold indicated by the handle diagram in Figure \ref{fig: Xna-kirby}, generalizing $Z_0 = X_{2,0}$. It is easy to verify that $X_{n,a}$ is a rational homology $4$-sphere with $\pi_1(X_{n,a}) \cong H_1(X_{n,a}) \cong \Z/n$, essentially the simplest construction of a manifold with those properties. The diffeomorphism type of $X_{n,a}$ depends only on $n$ and the parity of $a$. Thus, it is natural to ask a more concrete version of Question \ref{q: KZ2}:
\begin{question}
For $n>2$ and $a \in \{0,1\}$, does there exist a non-slice knot $K \subset S^3$ that is slice in $X_{n,a} - B^4$?
\end{question}
We invite the reader to find a knot with the needed properties.

\section{Proof of Theorem \ref{thm: main}} \label{sec: kawauchi}

Throughout this section, let $K \subset S^3$ be a SNACK. Up to equivariant isotopy, we may assume that $K$ is fixed setwise by the map $\phi \co S^3 \to S^3$ that is the restriction to $S^3$ of the linear involution $\Phi \co \R^4 \to \R^4$ given by $\Phi(x_1, x_2, x_3, x_4) = (x_1, -x_2, -x_3, -x_4)$.

Let $r \co \R^2 \to \R^2$ denote the reflection $r(x_1,x_2) = (x_1,-x_2)$; we also denote its restrictions to $D^2$ and $S^1$ by the same symbol. 
Let $\psi_K \co S^1 \to S^3$ denote the inclusion of $K$, chosen to be equivariant with respect to the involutions $r$ on $S^1$ and $\phi$ on $S^3$. (In particular, $\psi_K$ takes $(\pm1, 0)$ to the two fixed points of $\phi$.) By the equivariant tubular neighborhood theorem (see, e.g., \cite[Theorem 4.4]{Kankaanrinta}), we may extend $\psi_K$ to an embedding $\Psi_K \co S^1 \times D^2 \hookrightarrow S^3$ that parametrizes an equivariant closed tubular neighborhood of $K$, with the following properties:
\begin{itemize}
  \item $\Psi_K$ restricts to $\psi_K$ on $S^1 \times \{\vec 0\}$.
  \item For any $(x, y) \in S^1 \times D^2$, we have $\phi \circ \Psi_K (x, y) = \Psi_K(r(x),-y)$. 
  \item $\Psi_K$ determines the $0$-framing of $K$; that is, for any nonzero $y \in D^2$,  $\psi_K(S^1 \times \{y\})$ has linking number $0$ with $K$.
\end{itemize}

Let $X_K$ denote the $0$-trace of $K$, obtained by attaching a $0$-framed $2$-handle $D^2 \times D^2$ to $D^4$ using the attaching map $\Psi_K$. This manifold acquires an orientation from that of $D^4$. The boundary of $X_K$ is the $0$-surgery $S^3_0(K)$. The involution $\Phi|_{D^4}$ extends to an orientation-reversing involution $\Phi_K \co X_K \to X_K$, defined on the $2$-handle $D^2 \times D^2$ by $\Phi_K(x, y) = (r(x), -y)$. The fixed point set of $\Phi_K$ is a circle, consisting of the arcs $[-1,1] \times \{\vec 0\} \subset D^4$ and $([-1,1] \times \{0\}) \times \{\vec 0\} \subset D^2 \times D^2$. In particular, observe that $\Phi_K$ restricts to a fixed-point-free, orientation-reversing involution of $S^3_0(K)$, which we denote by $\phi_K$.


We now describe Kawauchi's construction (in slightly different terms). Let $Z_K$ denote the quotient $X_K/{\sim}$, where for all $x \in S^3_0(K)$, we set $x \sim \phi_K(x)$. Let $\pi\co X_K \to Z_K$ denote the quotient map. That is, we obtain $Z_K$ by a self-gluing of the boundary of $X_K$. Because $\phi_K$ has no fixed points, $Z_K$ is a smooth, closed $4$-manifold, and because $\phi_K$ is orientation-reversing, $Z_K$ naturally acquires an orientation from that of $X_K$.

\begin{lemma} \label{lemma: ZK-slice}
The knot $K$ is slice in $Z_K$.
\end{lemma}

\begin{proof}
Let $X'_K$ be the union of $X_K$ with an exterior collar $S^3_0(K) \times [0,1]$, attached along $S^3_0(K) \times \{1\}$, and let $Z'_K$ be the quotient of $X'_K$ by self-gluing by $\phi_K$ along $S^3_0(K) \times \{0\}$. Then clearly $X'_K \cong X_K$ and $Z'_K \cong Z_K$. Let $B \subset Z_K'$ denote the $0$-handle of $X_K$, which is still an embedded closed $4$-ball even after the gluing thanks to the collar. Let $V_K = Z_K' - \intt(B)$; there is a natural identification $\partial V_K = S^3$. Then $K$ bounds an embedded disk in $V_K$, namely the core of the $2$-handle of $X_K$. Thus, $K$ is slice in $Z_K$.
\end{proof}

\begin{remark}
In \cite{KawauchiRational}, Kawauchi considers the more general case of a strongly negative amphichiral knot $K$ in an arbitrary rational homology sphere $Y$, not just in $S^3$. He first considers $Y_0(K) \times [0,1]/ {\sim}$, where $(x,0) \sim (\phi_K(x),0)$, and proves that this is a rational homology $S^1 \times D^3$ bounded by $Y_0(K)$. Adding a $2$-handle along the meridian of $K$ then produces a rational homology ball bounded by $Y$, in which $K$ is slice. In the case where $Y=S^3$, this agrees with the description of $V_K$ in the previous paragraph.
\end{remark}

Let $DX_K$ denote the double of $X_K$: $DX_K = X_K \sqcup \overline{X_K}/{\sim}$, where the two copies are identified by the identity map of $S^3_0(K)$. This manifold acquires an orientation from that of $X_K$. Since $DX_K$ is the union of two simply-connected spaces along a connected intersection, it is simply-connected. Indeed, because $X_K$ is built with only a $0$- and $2$-handle (with even framing), it is well-known that $DX_K \cong S^2 \times S^2$, irrespective of $K$. (See, e.g., \cite[Corollary 5.1.6]{GompfStipsicz}.)

Let $\Pi_K \co DX_K \to Z_K$ be defined by $\pi$ on $X_K$ and by $\pi \circ \phi_K$ on $\overline {X_K}$. It is easy to see that $\Pi_K$ is a $2:1$ covering map, and hence it is the universal cover of $Z_K$. There is a nontrivial deck transformation $\tau_K \co DX_K \to DX_K$ that interchanges the two copies $X_K$ and $\overline{X_K}$ using $\Phi_K$. Using this covering map, we can deduce the algebraic topology of $Z_K$, as follows.

\begin{lemma} \label{lemma: ZK-algtop}
The manifold $Z_K$ is a rational homology $4$-sphere and has $\pi_1(Z_K) \cong H_1(Z_K) \cong H_2(Z_K) \cong  \Z_2$.
\end{lemma}

\begin{proof}
Since the universal cover of $Z_K$ is two-sheeted, we deduce that $\pi_1(Z_K) \cong H_1(Z_K) \cong \Z_2$ and hence $b_1(Z_K)=0$. The nontrivial element of $\pi_1(Z_K)$ can be given by any arc connecting two points in $S^3_0(K)$ that are exchanged by $\phi_K$.

To see that $Z_K$ is a rational homology sphere, we first note that $\chi(DX_K) = 2 \chi(X_K) - \chi(S^3_0(K) = 4$, and then $\chi(Z_K) = \chi(DX_K)/2 = 2$. Since $\chi(Z_K) = 2 - 2b_1(Z_K) + b_2(Z_K)$, we have $b_2(Z_K) = 0$. Universal coefficients and Poincar\'e duality then imply that $H_2(Z_K) \cong H^3(Z_K) \cong H_1(Z_K) \cong \Z_2$, as required.
\end{proof}

\begin{example}
Let $O$ denote the unknot; then $X_O \cong S^2 \times D^2$. To be explicit, let us identify $D^4$ with $D^2 \times D^2$, where the involution $\Phi$ is still given in coordinates by $\Phi(x_1, x_2,x_3,x_4) = (x_1,-x_2,-x_3, -x_4)$, and take $O$ to be $S^1 \times \{0\}$. The framing $\Psi_O$ is then just the inclusion of $S^1 \times D^2$. Then $X_O = (D^2 \times D^2) \cup (D^2 \times D^2)$, glued by the identity map of $S^1 \times D^2$. This is naturally identified as $(D^2 \cup_{S^1} D^2 ) \times D^2 = S^2 \times D^2$, and $S^3_0(O)$ is identified as $S^2 \times S^1$. By construction, $\Phi_O$ acts on each copy of $D^2 \times D^2$ by a reflection in the first factor and negation in the second. Thus, it acts on $S^2 \times D^2$ in the same fashion: a reflection $r \co S^2 \to S^2$ in the first factor and negation in the second factor.

Taking the double, we have $DX_O = S^2 \times (D^2 \cup _{S^1} D^2) = S^2 \times S^2$. The deck transformation $\tau_O$ acts by the reflection $r$ in the first factor, while interchanging the two copies of $D^2$ and negating in the second factor. That is, for $(x,y) \in S^2 \times S^2$, we have $\tau_O(x,y) = (r(x), -y)$. We thus see that $Z_O$ agrees with the construction of $Z_0$ in the introduction.
\end{example}

To prove Theorem \ref{thm: main}, we will use a $5$-dimensional argument (inspired by one of Mazur \cite{MazurContractible}) to show that the diffeomorphism $DX_K \cong S^2 \times S^2 = DX_O$ can be constructed equivariantly with respect to the deck transformations $\tau_K$ and $\tau_O$. Let $Q_K = X_K \times [-1,1]$. This is a $5$-manifold whose boundary is
\[
(X_K \times \{1\}) \cup (S^3_0(K) \times [-1,1]) \cup (X_K \times \{-1\}).
\]
Then $\partial Q(K)$ is naturally identified, after smoothing corners, with $DX_K$ (or, more precisely, with $DX_K'$ because of the collar). Define $\tilde \tau_K \co Q_K \to Q_K$ by $\tilde \tau_K(x,t) = (\Phi_K(x), -t)$. This is an involution of $Q_K$, and it restricts to $\tau_K$ on $\partial Q_K$. Continuing with the above example, we may identify $Q_O$ with $S^2 \times D^3$, where $\tilde \tau_O(x,y)  = (r(x), -y)$.

\begin{proposition} \label{prop: QK}
For any SNACK $K$, the pairs $(Q_K, \tilde \tau_K)$ and $(Q_O, \tilde \tau_O)$ are equivariantly diffeomorphic.
\end{proposition}

\begin{proof}
Note that $Q_K$ has a $5$-dimensional handle structure consisting of one $0$-handle and one $2$-handle, each of which is the product of the corresponding handle of $X_K$ with an interval, and the involution $\tilde \tau_K$ preserves this handle structure. After smoothing corners, we may identify the $0$-handle of $Q_K$ with $D^5$, and the $2$-handle with $D^2 \times D^3$, so that the involution $\tilde \tau_K$ is given on $D^5$ by
\[
\tilde \tau_K|_{D^5}(x_1, x_2, x_3, x_4, x_5) = (x_1, -x_2, -x_3, -x_4, -x_5).
\]
Since this is independent of $K$, we will omit the $K$ subscript denote this map by $\tilde\tau$. The attaching circle for the $2$-handle is $K \times \{0\}$, where we identify $S^3$ with $\partial D^5 \cap \{x_5 = 0\}$. The gluing map is an inclusion of $S^1 \times D^3$ into $\partial D^5$, parametrizing a $\tilde \tau$-invariant neighborhood of $K \times \{0\}$. We may likewise view the attaching circle for the 2-handle of $Q_O$, $O \times \{0\}$, as living in this same manifold.

By a theorem of Boyle and Chen \cite[Proposition 3.12]{BoyleChenHalf}, there is a homotopy from $K$ to $O$, equivariant with respect to our original involution $\phi \co S^3 \to S^3$, which is an isotopy except for finitely many pairs of simultaneous crossing changes. By slightly perturbing this in the $x_5$ direction, we may promote this to a $\tilde \tau$-equivariant isotopy taking $K \times \{0\}$ to $O \times \{0\}$ in $S^4$.

By the equivariant isotopy extension theorem (see, e.g., \cite[Theorem 8.6]{Kankaanrinta}), we may then find an equivariant ambient isotopy of $S^4$ taking $K \times \{0\}$ to $O \times \{0\}$. Under this isotopy, the framing of $K \times \{0\}$ used to define $Q_K$ induces a framing of $O \times \{0\}$, which \emph{a priori} may or may not agree with the framing of $O \times \{0\}$ used to define $Q_O$. However, note that a circle in $S^4$ only has two framings, which are distinguished by their surgeries: one framing yields $S^2 \times S^2$, while the other framing yields $S^2 \widetilde{\times} S^2 = \CP^2 \conn \CPbar^2$. Since we have already established that both $Q_K$ and $Q_O$ have boundary diffeomorphic to $S^2 \times S^2$, we deduce that the isotopy does indeed take the preferred framing of $K \times \{0\}$ to that of $O \times \{0\}$. Thus, the isotopy extends to an equivariant diffeomorphism from $Q_K$ to $Q_O$, as required.
\end{proof}

\begin{proof}[Proof of Theorem \ref{thm: main}]
Restricting the diffeomorphism from Proposition \ref{prop: QK} to the boundary gives an equivariant diffeomorphism $(DX_K,\tau_K) \cong (DX_O, \tau_O)$, and hence a diffeomorphism between the quotients, $Z_K \cong Z_O$. Thus, $K$ is slice in $Z_O$.
\end{proof}

\section*{Acknowledgements}

The author first learned the fact that the figure-eight knot is rationally slice from Tim Cochran, who very tactfully corrected an incorrect statement that the author made during a conference talk in Busan, South Korea, in 2014. This paper is dedicated to Tim's memory. The author is also grateful to Keegan Boyle, Jen Hom, Tye Lidman, JungHwan Park, Lisa Piccirillo, and Danny Ruberman for helpful conversations.

\bibliography{bibliography}

\newcommand{\etalchar}[1]{$^{#1}$}
\def\cprime{$'$} \def\MR{} \renewcommand{\MR}[1]{}
\providecommand{\bysame}{\leavevmode\hbox to3em{\hrulefill}\thinspace}
\providecommand{\MR}{\relax\ifhmode\unskip\space\fi MR }
\providecommand{\MRhref}[2]{%
  \href{http://www.ams.org/mathscinet-getitem?mr=#1}{#2}
}
\providecommand{\href}[2]{#2}
\begin{thebibliography}{DKM{\etalchar{+}}22}

\bibitem[BC22]{BoyleChenHalf}
Keegan Boyle and Wenzhao Chen, \emph{Negative amphichiral knots and the
  half-{C}onway polynomial}, \arxiv{2206.03598}, 2022.

\bibitem[Cha07]{ChaRational}
Jae~Choon Cha, \emph{The structure of the rational concordance group of knots},
  Mem. Amer. Math. Soc. \textbf{189} (2007), no.~885, x+95. \MR{2343079}

\bibitem[CLR08]{ChaLivingstonRuberman}
Jae~Choon Cha, Charles Livingston, and Daniel Ruberman, \emph{Algebraic and
  {H}eegaard-{F}loer invariants of knots with slice {B}ing doubles}, Math.
  Proc. Cambridge Philos. Soc. \textbf{144} (2008), no.~2, 403--410.

\bibitem[DKM{\etalchar{+}}22]{DKMPSCable}
Irving Dai, Sungkyung Kang, Abhishek Mallick, JungHwan Park, and Matthew
  Stoffregen, \emph{The $(2,1)$-cable of the figure-eight knot is not smoothly
  slice}, \arxiv{2207.14187}, 2022.

\bibitem[GS99]{GompfStipsicz}
Robert~E. Gompf and Andr\'as~I. Stipsicz, \emph{{$4$}-manifolds and {K}irby
  calculus}, Graduate Studies in Mathematics, vol.~20, American Mathematical
  Society, Providence, RI, 1999. \MR{1707327}

\bibitem[Hil02]{HillmanGeometries}
J.~A. Hillman, \emph{Four-manifolds, geometries and knots}, Geometry \&
  Topology Monographs, vol.~5, Geometry \& Topology Publications, Coventry,
  2002. \MR{1943724}

\bibitem[HKL16]{HeddenKimLivingston}
Matthew Hedden, Se-Goo Kim, and Charles Livingston, \emph{Topologically slice
  knots of smooth concordance order two}, J. Differential Geom. \textbf{102}
  (2016), no.~3, 353--393. \MR{3466802}

\bibitem[HKPS20]{HomKangParkStoffregen}
Jennifer Hom, Sungkyung Kang, Junghwan Park, and Matthew Stoffregen,
  \emph{Linear independence of rationally slice knots}, \arxiv{2011.07659},
  2020.

\bibitem[HM17]{HendricksManolescuInvolutive}
Kristen Hendricks and Ciprian Manolescu, \emph{Involutive {H}eegaard {F}loer
  homology}, Duke Math. J. \textbf{166} (2017), no.~7, 1211--1299.

\bibitem[Kan07]{Kankaanrinta}
Marja Kankaanrinta, \emph{Equivariant collaring, tubular neighbourhood and
  gluing theorems for proper {L}ie group actions}, Algebr. Geom. Topol.
  \textbf{7} (2007), 1--27. \MR{2289802}

\bibitem[Kaw80]{KawauchiFigureEight}
Akio Kawauchi, \emph{The $(1,2)$ cable of the figure eight knot is rationally
  slice}, unpublished,
  \url{http://www.sci.osaka-cu.ac.jp/OCAMI/kawauchi/RationalSliceof%208-knot.pdf},
  1980.

\bibitem[Kaw09]{KawauchiRational}
\bysame, \emph{Rational-slice knots via strongly negative-amphicheiral knots},
  Commun. Math. Res. \textbf{25} (2009), no.~2, 177--192. \MR{2554510}

\bibitem[KR20]{KlugRuppikDeep}
Michael~R. Klug and Benjamin~M. Ruppik, \emph{Deep and shallow slice knots in
  4-manifolds}, \arxiv{2009.03053}, 2020.

\bibitem[KW18]{KimWuRational}
Min~Hoon Kim and Zhongtao Wu, \emph{On rational sliceness of {M}iyazaki's
  fibered, $-$amphicheiral knots}, Bull. Lond. Math. Soc. \textbf{50} (2018),
  no.~3, 462--476. \MR{3829733}

\bibitem[Lev69]{JLevineCobordism}
Jerome Levine, \emph{Knot cobordism groups in codimension two}, Comment. Math.
  Helv. \textbf{44} (1969), 229--244.

\bibitem[Maz61]{MazurContractible}
Barry Mazur, \emph{A note on some contractible {$4$}-manifolds}, Ann. of Math.
  (2) \textbf{73} (1961), 221--228.

\bibitem[OS03a]{OSzAbsolute}
Peter~S. Ozsv{\'a}th and Zolt{\'a}n Szab{\'o}, \emph{Absolutely graded {F}loer
  homologies and intersection forms for four-manifolds with boundary}, Adv.
  Math. \textbf{173} (2003), no.~2, 179--261. \MR{1957829}

\bibitem[OS03b]{OSz4Genus}
\bysame, \emph{Knot {F}loer homology and the four-ball genus}, Geom. Topol.
  \textbf{7} (2003), 615--639. \MR{2026543}

\bibitem[OSS14]{OSSzHomomorphisms}
Peter~S. Ozsv{\'a}th, Andr{\'a}s Stipsicz, and Zolt{\'a}n Szab{\'o},
  \emph{Concordance homomorphisms from knot {F}loer homology},
  \arxiv{1407.1795}, 2014.

\bibitem[Ras10]{RasmussenGenus}
Jacob Rasmussen, \emph{Khovanov homology and the slice genus}, Invent. Math.
  \textbf{182} (2010), no.~2, 419--447.

\end{thebibliography}
\bibliographystyle{amsalpha}

\end{document}